\documentclass[11pt]{article}
\usepackage{graphicx, times}
\usepackage{amsfonts}
\usepackage{amsmath}
\usepackage{amsthm}
\usepackage{amssymb}
\usepackage{fancyhdr}
\usepackage{indentfirst}
\usepackage{titlesec}
\usepackage{hyperref}
\usepackage{abstract}
\usepackage{color}

\usepackage[normalem]{ulem}

\newcommand{\N}{\mathbb{N}}
\newcommand{\mH}{\mathcal{H}}

\newcommand{\lan}{\langle}
\newcommand{\ran}{\rangle}
\newcommand{\la}{\lambda}

\newcommand{\La}{\Lambda}

\newcommand{\Si}{\Sigma}
\newcommand{\de}{\delta}

\newcommand{\pr}{\prime}

\newcommand{\ga}{\gamma}

\newcommand{\lap}{\triangle}
\newcommand{\ti}{\tilde}

\newcommand{\mL}{\mathcal{L}}

\newcommand{\Z}{\mathcal{Z}}

\begin{document}

\newtheorem{theorem}{Theorem}[section]
\newtheorem{proposition}[theorem]{Proposition}
\newtheorem{corollary}[theorem]{Corollary}

\newtheorem{claim}{Claim}

\theoremstyle{remark}
\newtheorem{remark}[theorem]{Remark}

\theoremstyle{definition}
\newtheorem{definition}[theorem]{Definition}

\theoremstyle{plain}
\newtheorem{lemma}[theorem]{Lemma}

\numberwithin{equation}{section}

 \titleformat{\section}
   {\normalfont\bfseries\large}
   {\arabic{section}}
   {12pt}{}
 \titleformat{\subsection}[runin]
   {\normalfont\bfseries}
   {\arabic{section}.\arabic{subsection}}
   {11pt}{}
%
%
%

\pagestyle{headings}
\renewcommand{\headrulewidth}{0.4pt}

\title{\textbf{Sweeping out 3-manifold of positive Ricci curvature by short 1-cycles via estimates of min-max surfaces}}
\author{Yevgeny Liokumovich and Xin Zhou\footnote{The second author is partially supported by NSF grant DMS-1406337.}}

\maketitle

\pdfbookmark[0]{}{beg}

\renewcommand{\abstractname}{}    
\renewcommand{\absnamepos}{empty} 
\begin{abstract}
\textbf{Abstract:} 
We prove that given a three manifold with an arbitrary metric $(M^3, g)$ of positive Ricci curvature, there exists a sweepout 
of $M$ by surfaces of genus $\leq 3$ and areas bounded by $C vol(M^3, g)^{2/3}$.
We use this result to construct a sweepout of $M$ by 1-cycles of length
at most $C vol(M^3, g)^{1/3}$, giving a partial answer to
a question of L. Guth.

The sweepout of surfaces is generated from a min-max minimal surface. If further assuming a positive scalar curvature lower bound, we can get a diameter upper bound for the min-max surface.
\end{abstract}


\section{Introduction}

Let $M$ be a 3-manifold with positive Ricci curvature.
In this paper we obtain quantitative results about sweepouts of $M$
by 1-cycles and surfaces.

\begin{theorem} \label{MAIN_CYCLE}
Every closed 3-manifold $M$ of positive Ricci curvature
admits a map $f: M \rightarrow \mathbb{R}^2$
with fibers of length at most
$C Vol(M)^{\frac{1}{3}}$ for a universal constant $C>0$.
\end{theorem}

Moreover, the family of 1-cycles 
$\{ f^{-1} (x) \}_{x\in \mathbb{R}^2}$
is continuous in a strong sense. We define this precisely
in Section 3.

The geometric invariant that we bound from above
in Theorem \ref{MAIN_CYCLE}
is called the k-waist of a Riemannian manifold. 
For an n-dimensional manifold $M$ and $k<n$
the k-waist, $waist_k(M)$, is defined as 
 $\inf \{ \sup_{x \in \mathbb{R}^{n-k}} \{Vol_k f^{-1}(x) \} \}$, where the infimum is taken over
all proper functions $f: M \rightarrow \mathbb{R}^{n-k}$.
Waists have been defined by Gromov and extensively 
studied in \cite{Gro83}, \cite{Gro88}, \cite{Gro03}, 
\cite{Gro09}, \cite{Gro10}, \cite{Gro15}.
Gromov proved deep theorems about waists of manifolds
(see \cite{Gu14} and \cite{M11} for exposition of some of his work), 
yet many significant questions about waists remain open.
In particular, existence of upper bounds for 
$waist_k(M)$ when $k<n-1$ is a widely open question.
In \cite{Gu10} Guth asked if for every Riemannian metric $g$ on a 3-torus $T$
the 1-waist $waist_1(T,g)$ is bounded from above by $C Vol(T,g)^{\frac{1}{3}}$.
More generally, one may ask analogous questions
for any Riemannian 3-manifold, that is,
does there exist a constant $C$, such that for any metric $g$ there exists 
a map $f: M \rightarrow \mathbb{R}^2$ with fibers 
of length at most $C Vol(M,g)^{\frac{1}{3}}$?
If this is true, it would be a strong generalization of the systolic inequality. 
In Theorem \ref{MAIN_CYCLE} we affirmatively answer this question for 3-manifolds $M$ under an additional assumption of $Ric>0$. It is to our knowledge the first occasion where such a generalization is proved.

Theorem \ref{MAIN_CYCLE} generalizes, in the setting of positive 
Ricci curvature,
some fundamental theorems about
the length of closed geodesics and stationary one-cycles
in Riemannian manifolds due to Gromov and Nabutovsky-Rotman.
As consequence of Theorem \ref{MAIN_CYCLE}
we obtain

\begin{corollary} \label{SYSTOLIC}
Let $M$ be a closed 3-manifold of positive Ricci curvature.
If $M$ is homeomorphic to $\mathbb{R}P^3$ then it contains a non-contractible
closed geodesic of length at most $C Vol(M) ^{\frac{1}{3}}$.
If $M= (S^3,g)$ then it contains a geodesic net of length
at most $C Vol(M) ^{\frac{1}{3}}$.
\end{corollary}

The first statement of 
Corollary \ref{SYSTOLIC} is a special case of 
Gromov's systolic inequality \cite{Gro83}. 
The systole of a Riemannian manifold is defined as the length of the shortest non-contractible geodesic loop. 
In the same paper Gromov conjectured that 
every manifold contains a non-trivial closed geodesic of length at most
$C_n Vol(M)^{1/n}$. 
Nabutovsky and Rotman \cite{NR04}  proved that
every Riemannian manifold contains a 
stationary 1-cycle of length at most $C_n Vol(M)^{\frac{1}{n}}$
(a stationary 1-cycle need not be a closed geodesic;
it may look, for example, like a bouquet of geodesic loops
all intersecting at a point with tangent vectors at that point
summing up to $0$, see \cite{NR04} for more examples).

If $M$ is topologically a sphere, then a min-max argument
yields an upper bound for the length of a stationary 1-cycle, 
giving an alternative proof of a special case of 
the result of Nabutovsky and Rotman.


Here we present a short (and incomplete) overview of previously known
estimates for sweepouts of manifolds.
In \cite{Gu07} Guth proved that every open subset of Euclidean space
$U \subset \mathbb{R}^n$ admits a sweepout by relative $k$-cycles
of volume at most $C_n Vol(U)^{\frac{k}{n}}$.
In general such inequalities do not hold for Riemannian manifolds (see Appendix 5 in \cite{Gu07} and \cite{PS15}).
However, we may control volumes of $(n-1)$-cycles if we
impose an additional requirement on the metric.
In \cite{GL}, among other results,
it was shown that if $M$ is conformally equivalent
to a manifold with non-negative Ricci curvature
then it admits a sweepout by $(n-1)$-cycles of volume
at most $C_n Vol(U)^{\frac{n-1}{n}}$ (in \cite{S15} 
Sabourau independently constructed a
sweepout of $M$ with $Ric(M)\geq 0$ by $(n-1)$-cycles
of controlled volume).\\

When $Ric>0$ and $n=3$ we show that we can simultaneously
control the area and the genus of surfaces in the sweepout, which will be essential in the proof of Theorem \ref{MAIN_CYCLE}.

\begin{theorem}\label{main theorem3}
Given a three manifold with an arbitrary metic $(M^3, g)$ of positive Ricci curvature, i.e. $Ric_g>0$, there exists a minimal surface $\Si^2_0$, such that $Area(\Si_0)\leq C vol(M^3, g)^{2/3}$, for a universal constant $C>0$. Also we have
\begin{itemize}
\vspace{-5pt}
\setlength{\itemindent}{1em}
\addtolength{\itemsep}{-0.7em}
\item If $\Si_0$ is orientable, then the genus $g_0$ of $\Si$ satisfies $g_0\leq 3$, and there exists a smooth sweepout $\{\Si_t\}_{t\in[-1, 1]}$ of $(M^3, g)$, such that

\begin{itemize}
\vspace{-5pt}

\item $\{\Si_t\}$ forms a Heegaard splitting of $M^3$, i.e. $\Si_t$ is an embedded surface of genus $g_0$, for $t\in(-1, 1)$, and $\Si_{-1}$ and $\Si_1$ are graphs;
\item $Area(\Si_t)<Area(\Si_0)$ for $t\neq 0$.
\end{itemize}

\item  If $\Si_0$ is non-orientable, then the genus $\ti{g}_0$ of its double cover $\ti{\Si}_0$ satisfies $\ti{g_0}\leq 3$.  Moreover, by removing $\Si$ from $M$, we get a manifold with boundary $\ti{M}$ with $\partial\ti{M}=\ti{\Si}_0$, and there exists a smooth sweepout $\{\Si_t\}_{t\in[0, 1]}$ of $\ti{M}$, such that
\begin{itemize}
\vspace{-5pt}

\item $\{\Si_t\}$ forms a Heegaard splitting of $\ti{M}$, i.e. $\Si_t$ is an embedded surface of genus $g_0$ lying in the interior of $\ti{M}$, for $t\in(0, 1)$, and $\Si_0=\partial\ti{M}$;
\item $Area(\Si_t)<2 Area(\Si_0)$ for $t\neq 0$.
\end{itemize}
\end{itemize} 
\end{theorem}

\begin{remark}
After this paper was submitted for publication 
Ketover, Marques and Neves \cite{KMN} proved that the min-max minimal surface $\Si_0$ in 3-manifolds with positive Ricci curvature must
be orientable. This result follows from a clever application
of the catenoid estimate they derive.
Therefore, the second case of Theorem 1.3 does not
occur. Note that this
simplifies our construction of a continuous 
sweepout of $M$ by 1-cycles as it rules out Case 2
in the proof of Theorem \ref{main_cycles1} (see Section 5).
\end{remark}

We would like to compare our result with that of F. Marques and A. Neves \cite{MN11}. In \cite{MN11}, assuming $Ric_g>0$ and the scalar curvature lower bound $Scal_g\geq 6$, Marques-Neves produced a smooth sweepout $\{\Si_t\}_{t\in[0, 1]}$, where the genus of $\Si_t$ is the Heegaard genus\footnote{Heegaard genus is the least genus of a Heegaard surface; so Heegaard genus is the best one we can expect in Theorem \ref{main theorem3}.}, and $Area(\Si_t)\leq 4\pi$. The advantage of \cite{MN11} is that they have better estimates for the genus. However, from the point of view of area estimates (e.g. for the application to prove Theorem \ref{MAIN_CYCLE}), our result can be much better than that in \cite{MN11} while we still have a relatively good genus estimate. An example illustrating this fact is a long and thin 3-dimensional ellipsoid; when we normalize the scalar curvature lower bound to be $6$, the width can be very small (compared to $4\pi$). The difference between our method with \cite{MN11} is that we use the Almgren-Pitts min-max theory \cite{AF62, P81} for general sweepouts constructed in \cite{GL}, while Marques-Neves used the Colding-De Lellis \cite{CD03} (or Simon-Smith \cite{Sm82}) min-max method for smooth sweepouts given by Heegaard splittings. We refer to \S \ref{open questions} for more discussion.\\

The sweepout $\{\Si_t\}$ in Theorem \ref{main theorem3} is used to construct a sweepout by 1-cycles of controlled length in Theorem \ref{MAIN_CYCLE}. An important open question is
whether one can construct a sweepout by \emph{closed curves} of controlled length rather than 
1-cycle (see more discussion in \S \ref{open questions}). 
One approach in this direction is to 
first construct a sweepout of $M$ by spheres
or tori of controlled area and \emph{diameter}.
For this purpose, we derive the following partial result. In particular, if we further assume a scalar curvature lower bound, we can get a uniform diameter upper bound for the min-max minimal surface.

\begin{theorem}\label{main theorem3.1}
Let $(M^3, g)$ be as in Theorem \ref{main theorem3}; if the scalar curvature of $(M^3, g)$ is bounded from below, i.e. $Scal_g\geq 2\La$, for some $\La>0$, then the diameter of $\Si^2_0$ (when it is orientable) or the diameter of its double cover (when it is non-orientable) is bounded from above by $\sqrt{6}\frac{\pi}{\sqrt{\La}}$. 
\end{theorem}




\vspace{1em}
The main idea of proving Theorem \ref{MAIN_CYCLE} is a dimension reduction type argument. We first construct a nice sweepout by 2-surfaces with controlled area and genus by Theorem \ref{main theorem3}. Then we continuously sweep out these 2-surfaces by 1-cycles. A large portion of the argument (in particular,
proof of Lemma \ref{main_cycles1}) is devoted to 
making this family continuous in a strong sense (cf. Section \ref{intro to 1-cycles}). This continuity is needed 
to obtain a geodesic net of controlled length as a limit 
of a min-max sequence of 1-cycles. 

Theorem \ref{main theorem3} is proved by combining several ingredients. We apply the Almgren-Pitts min-max theory to the sweepout constructed in \cite{GL} and get a min-max minimal surface of controlled area. By using one of the authors Morse index bound \cite{Z12}, we can get the desired genus bound via Schoen-Yau genus estimates \cite{Y87}. The existence of good Heegaard splitting follows from Meeks-Simon-Yau \cite{MSY}. The diameter estimates (Theorem \ref{main theorem3.1}) for the min-max surface comes from Schoen-Yau diameter estimates \cite{SY83} and the Morse index estimate.

\vspace{1em}
Our paper is organized as follows. In \S \ref{area and diameter estimates}, we prove Theorem \ref{main theorem3} and Theorem \ref{main theorem3.1}.
In \S \ref{intro to 1-cycles} we give a precise definition of the sweepout
by 1-cycles. In \S \ref{parametric sweepouts of surfaces}, we show how to sweep out a family of surfaces simultaneously by continuous 1-cycles with lengths controlled by the genus and area. In \S \ref{proof of main_cycle} we prove Theorem \ref{MAIN_CYCLE} by combining results in \S \ref{area and diameter estimates} and \S \ref{parametric sweepouts of surfaces}. Finally, we summarize several interesting open questions in \S \ref{open questions}.

\vspace{1em}
{\bf Acknowledgements}: Both authors would like to thank Larry Guth for getting them together and useful comments. Y.L. would like to thank Alexander Nabutovsky and Regina Rotman for helpful discussions
and Alexander Nabutovksy for pointing out a mistake in the earlier
draft.

\section{Area and diameter estimates for the min-max minimal surface}\label{area and diameter estimates}

We outline the proof of Theorem \ref{main theorem3}.
\begin{proof}
Since $(M, g)$ has positive Ricci curvature, we can start with a sweepout constructed in \cite{GL}, 
and, independently, \cite{S15}, i.e. $\Phi: S^1\rightarrow \Z_2(M^3)$, such that for a universal constant $C>0$,
$$\sup_{t\in S^1}\mH^2(\Phi(t))\leq C vol(M^3, g)^{2/3}.$$
Now we can adapt such a sweepout to the Almgren-Pitts theory \cite{AF62, AF65, P81} as in \cite{MN12, Z12}. Application of the Almgren-Pitts theory produces a min-max minimal surface $\overline{\Si}_0$, such that an integer multiple $n_0 \overline{\Si}_0$, $n_0\in\N$ achieves the min-max value $W$, i.e.
$$n_0 Area(\overline{\Si}_0)=W\leq \sup_{t\in S^1}\mH^2(\Phi(t))\leq C vol(M^3, g)^{2/3}.$$

Using \cite[Theorem 1.1]{Z12}, there exists a minimal surface $\Si_0$, such that $\Si_0$ has least area among all closed, embedded, minimal hypersurfaces in the following sense. Define (see \cite[(1.1)]{Z12})
\begin{displaymath}\label{minimal volume of embedded minimal hypersurfaces}
\left. W_M= \inf\Big\{ \begin{array}{ll}
Area(\Si), \quad \textrm{ if $\Si$ is an orientable minimal surface}\\
2Area(\Si), \quad \textrm{ if $\Si$ is a non-orientable minimal surface}
\end{array} \Big\}.\right. 
\end{displaymath}
Then $Area(\Si_0)=W_M$ when it is orientable, or $2Area(\Si_0)=W_M$ when it is non-orientable. 

By comparing the area of $\overline{\Si}_0$ with that of $\Si_0$, we have that:
\begin{itemize}
\setlength{\itemindent}{1em}
\addtolength{\itemsep}{-0.7em}
\item $Area(\Si_0)\leq 2Area(\overline{\Si}_0)\leq C vol(M^3, g)^{2/3}$.
\end{itemize}
Moreover, when $\Si_0$ is orientable, it is proven in \cite[Theorem 1.1]{Z12} that
\begin{itemize}
\setlength{\itemindent}{1em}
\addtolength{\itemsep}{-0.7em}
\item The Morse index of $\Si_0$ is one.
\end{itemize}
When $\Si_0$ is non-orientable, it is shown by \cite[Theorem A]{MR15} that
\begin{itemize}
\setlength{\itemindent}{1em}
\addtolength{\itemsep}{-0.7em}
\item The Morse index of the double cover $\ti{\Si}_0$ of $\Si_0$ is one.
\end{itemize}
\begin{remark}
Our definition of $W_M$ is the same as $\mathcal{A}_1(M)$ in \cite{MR15}; and when $Ric_M>0$, \cite[Theorem A]{MR15} reduces to \cite[Theorem 1.1]{Z12}, except that \cite[Theorem A]{MR15} showed that the double cover of a non-orientable min-max surface has Morse index 1.
\end{remark}

By \cite[\S 4]{Y87}\cite[Theorem 2]{F88}, when $Ric_g>0$, The Morse index equal to one implies that the genus $g_0=g(\Si_0)$ of $\Si_0$ (when it is orientable) or the genus $\ti{g}_0=g(\ti{\Si_0})$ of its double cover $\ti{\Si}_0$ (when it is non-orientable) is bounded by $3$. \footnote{It is conjectured that the optimal upper bound is 2. See also \cite{N14}.}

Next we show the existence of good sweepouts generated by $\Si_0$. When $\Si_0$ is orientable, 
we claim that $\Si_0$ must be a Heegaard splitting; or equivalently, $M^3\backslash \Si_0$ is a union of two connected component $M_+$ and $ M_-$, such that $M_+$ and $M_-$ are both handlebodies. This is essentially due to \cite{MSY} (see also \cite[Lemma 3.2]{MN11}). In fact, $\Si_0$ separates $M$ into two connected components $M_+$ and $M_-$ by \cite[Proposition 3.5]{Z12} (see also \cite[Lemma 3.2]{MN11}). By minimizing area in the isotopy class of $\Si_0$ inside $M_+$ using \cite[Theorem 1']{MSY}, we either get another minimal surface $\Si^{\pr}$ in the interior of $M_+$, or we get an empty set, i.e $\Si_0$ can be isotopically changed to a surface of arbitrarily small area. The first case will violate the Frankel's Theorem \cite{F66} which says that every two closed minimal surfaces must intersect when $Ric_g>0$; while the second case implies that $M_+$ is a handlebody by \cite[Proposition 1]{MSY}. Similarly $M_-$ is also a handlebody. By \cite[Lemma 3.5]{MN11}\footnote{The construction there only used the fact that there is no non-intersecting minimal surfaces, and it is true here by Frankel's Theorem \cite{F66} as $Ric_g>0$.}, we can construct a heegaard splitting $\{\Si_t\}_{t\in[-1, 1]}$ satisfying the requirement of Theorem \ref{main theorem3}.

When $\Si_0$ is non-orientable, by removing $\Si_0$ from $M$, we get a manifold $\ti{M}$ with boundary $\partial{M}$. The boundary $\partial{M}$ is a minimal surface, and is a double cover of $\Si_0$. Similar argument as above shows that $\ti{M}$ is a handlebody. Again by the same method in \cite[Lemma 3.5]{MN11}, we can construct a Heegaard splitting $\{\Si_t\}_{t\in[0, 1]}$ satisfying the requirement.
\end{proof}

To prove Theorem \ref{main theorem3.1}, we will adapt the Schoen-Yau \cite{SY83} diameter estimates via scalar curvature lower bound for stable minimal surfaces as follows. Let $\Si$ be a two-sided minimal surface possibly with boundary. Here ``two-sided" means that $\Si$ has a unit normal vector field $\nu$. Given a function $u\in C^1_0(\Si)$, the second variation of area functional along normal deformation in the direction of $u(x)\nu(x)$ is given by \cite[Chap. 1, \S 8]{CM11}:
$$\de^2\Si(u, u)=\int_{\Si}|\nabla_{\Si} u|^2-\big(Ric(\nu, \nu)+|A|^2\big)u^2 d\mu=-\int_{\Si}uL_{\Si} u d\mu,$$
where $L_{\Si}u=\lap_{\Si}u+\big(Ric(\nu, \nu)+|A|^2\big)u$ is the Jacobi operator, and $A$ is the second fundamental form of $\Si$. $\Si$ is stable if $\de^2\Si(u, u)\geq 0$ for all $u\in C^1_0(\Si)$.
\begin{proposition}\label{diameter estimates for stable surface}
Given a three manifold $(M^3, g)$, assume that the scalar curvature is bounded from below $R^M\geq 2\La$, $\La>0$. Let $\Si$ be a two-sided stable minimal surface with boundary $\partial \Si$, then the inf-radius $\rho(\Si)$ of $\Si$ is bounded from below by $\sqrt{\frac{3}{2}}\frac{\pi}{\sqrt{\La}}$.\footnote{When preparing the manuscript, the authors learned that A. Carlotto also did something similar \cite[Proposition 2.12]{Ca15}.}
\end{proposition}
\begin{proof}
The fact that $\Si$ is stable implies that the Jacobi operator $L_{\Si}$ is non-positive. Let $\varphi$ be the first Dirichlet eigenfunction of $L_{\Si}$, i.e. $L_{\Si}\varphi =-\la \varphi$, $\la \geq 0$. Then $\varphi >0$ in the interior of $\Si$. Using \cite[page 193]{SY79}, by rewriting the Jacobi operator, we have
\begin{equation}\label{Jacobi operator 1}
L_{\Si}\varphi=\lap_{\Si}\varphi+\frac{1}{2}(R^M-R^{\Si}+|A|^2)\varphi=-\la\varphi\leq 0,
\end{equation}
where $R^M$ and $R^{\Si}$ are respectively the scalar curvatures of $M$ and $\Si$. 

Take a point $p$ in the interior of $\Si$, such that the distance of $p$ to $\partial\Si$ achieves the inf-radius $\rho(\Si)$ of $\Si$. Now consider the problem of minimizing the following functional
$$\mL(\ti{\ga})=\int_{\ti{\ga}}\varphi ds,$$
among all curves $\ti{\ga}$ connecting $p$ to $\partial\Si$. Assume that $\ga$ achieves a minimum, then
$$\int_{\ga}ds=Length(\ga)\geq \rho(\Si).$$

The first variation of $\mL$ at $\ga$ vanishes:
\begin{equation}\label{weighted first variation}
\de\mL(\ga)=\int_{\ga}\lan\nabla\varphi(\ga(s)), V(s)\ran ds+\int_{\ga}\varphi(\ga(s))\lan\nabla_v V(s), v\ran ds=0.
\end{equation}
Here $\nabla$ is the Riemannian connection of $\Si$, and $V(s)$ is an arbitrary variational vector field along $\ga(s)$ vanishing at the end points of $\ga$, and $v(s)$ in the unit tangent vector field along $\ga(s)$, and $ds$ is the length parameter. Integrating by parts shows that
$$\int_{\ga}\lan V(s), (\nabla\varphi )^{\perp}-\varphi(\ga(s))\nabla_v v \ran ds=0,$$
where $(\nabla \varphi)^{\perp}$ is normal component (with respect to the tangent vector of $\ga$ in $\Si$) of $\nabla \varphi$.  
Therefore the weighted geodesic equation is
\begin{equation}\label{weighted geodesic equation}
\varphi(\ga(s))\nabla_v v-(\nabla \varphi)^{\perp}=0.
\end{equation}

The second variation of $\mL$ is non-negative:
\begin{displaymath}
\begin{split}
\de^2\mL(\ga) & =\int_{\ga}\big(Hess\varphi(V, V)+\lan\nabla\varphi, \nabla_V V\ran\big)ds+2\int_{\ga}\lan\nabla\varphi, V\ran\lan\nabla_v V, v\ran ds\\
                       & +\int_{\ga}\varphi\big(\lan\nabla_v\nabla_V V, v\ran-K^{\Si}(V, v, V, v)+\lan\nabla_v V, \nabla_v V\ran-\lan\nabla_v V, v\ran^2\big)ds\geq 0.
\end{split}
\end{displaymath}
Here $K^{\Si}$ is the curvature tensor of $\Si$. Denote $\nu$ by the unit normal vector field along $\ga$, and let $V(s)=f(s)\nu(s)$ for some function $f$ which vanishes at the end points of $\ga$. Using (\ref{weighted first variation}) and (\ref{weighted geodesic equation}) we have,
\begin{equation}\label{weighted second variation}
\begin{split}
\de^2\mL(\ga) & = \int_{\ga}\big[Hess\varphi(\nu, \nu)-2\varphi(\ga(s))\lan\nabla_v\nu, v\ran^2\big]f^2ds\\
                       & +\int_{\ga}\varphi(\ga(s))\big(|\nabla_v f|^2-K^{\Si}(\nu, v, \nu, v)f^2\big)ds \geq 0.
\end{split}
\end{equation}
Using the fact that $\lap_{\Si}\varphi=Hess\varphi(\nu, \nu)+Hess\varphi(v, v)=Hess\varphi(\nu, \nu)+vv\varphi-\nabla_{\nabla_v v}\varphi$ and (\ref{weighted geodesic equation}),  (\ref{Jacobi operator 1}) can be re-written as
$$Hess\varphi(\nu, \nu)-\lan\nabla_v v, \nu\ran^2\varphi-\frac{1}{2}R^{\Si}\varphi\leq -vv\varphi-\frac{1}{2}(R^M+|A|^2)\varphi.$$
Combing this with (\ref{weighted second variation}) and using the fact that $\varphi>0$ and that $R^M\geq 2\La$, we have,
$$\int_{\ga}\varphi |\nabla_v f|^2-\La\varphi f^2- vv\varphi f^2 ds\geq 0.$$
Now parametrize $\ga$ by the length parameter on $[0, l]$, with $l=length(\ga)$, and using integration by part, we get
$$\int_0^l -\varphi f f^{\prime\prime}-f\varphi^{\prime} f^{\prime}-(\La\varphi+\varphi^{\prime\prime})f^2 ds \geq 0.$$
This implies that the following operator $L_0$ is non-negative\footnote{$L_0$ is the same as that in \cite[p577]{SY83} where the $"f"$ used in \cite{SY83} is a constant in our setting.}:
$$L_0 f=-\frac{d^2 f}{ds^2}-\frac{1}{\varphi}\frac{d \varphi}{ds}\frac{d f}{ds}-(\La+\frac{1}{\varphi}\frac{d^2 \varphi}{ds^2})f.$$
Let $h(s)$ be the first Dirichlet eigen-function of $L_0$ on $[0, l]$, then $h(s)>0$, and
$$\frac{h^{\prime\prime}}{h}+\frac{\varphi^{\prime}}{\varphi}\frac{h^{\prime}}{h}+\La+\frac{\varphi^{\prime\prime}}{\varphi}\leq 0.$$
Multiply the above inequality with any $f^2$, $f\in C^1_0([0, l])$, and use integration by part, then
$$\int_0^l \frac{(h^{\prime})^2}{h^2}f^2-2\frac{h^{\prime}}{h}f f^{\prime}+\frac{(\varphi^{\prime})^2}{\varphi^2}f^2-2\frac{\varphi^{\prime}}{\varphi}f f^{\prime}+\frac{\varphi^{\prime}}{\varphi}\frac{h^{\prime}}{h} f^2+\La f^2 ds \leq 0.$$
Re-arranging, we get
$$\int_0^l \frac{1}{2}(\frac{h^{\prime}}{h}+\frac{\varphi^{\prime}}{\varphi})^2f^2+\frac{1}{2}\Big(\frac{(h^{\prime})^2}{h^2}+\frac{(\varphi^{\prime})^2}{\varphi^2}\Big)f^2+\La f^2 ds\leq 2\int_0^l f f^{\prime}(\frac{h^{\prime}}{h}+\frac{\varphi^{\prime}}{\varphi}) ds.$$
By the Cauchy-Schwartz inequality,
$$2f f^{\prime}(\frac{h^{\prime}}{h}+\frac{\varphi^{\prime}}{\varphi})\leq \frac{1}{2}(\frac{h^{\prime}}{h}+\frac{\varphi^{\prime}}{\varphi})^2f^2+\frac{1}{2}\Big(\frac{(h^{\prime})^2}{h^2}+\frac{(\varphi^{\prime})^2}{\varphi^2}\Big)f^2+\frac{3}{2}(f^{\prime})^2.$$
So
$$\int_0^l \La f^2 ds \leq \frac{3}{2}\int_0^l (f^{\prime})^2 ds.$$
It implies that the operator $-\frac{d^2}{ds^2}-\frac{2}{3}\La$ is non-negative on $[0, l]$, so ODE comparison implies that
$$l\leq \pi/\sqrt{\frac{2}{3}\La}=\sqrt{\frac{3}{2}}\frac{\pi}{\sqrt{\La}}.$$
\end{proof}

Next we prove Theorem \ref{main theorem3.1}. Let us first assume that $\Si_0$ is orientable, and hence is two-sided (c.f. \cite[Proposition 3.5]{Z12}). Pick two points $p$, $q$ on $\Si$ such that the distance $d=d(p, q)$ achieves the diameter. Consider the geodesic balls $B(p, d/2)$ and $B(q, d/2)$ of $\Si$. As the Morse index of $\Si_0$ is one, at least one geodesic ball, say $B(p, d/2)$ is a stable minimal surface with smooth boundary. Then the proof of Proposition \ref{diameter estimates for stable surface} implies that $d/2\leq \sqrt{\frac{3}{2}}\frac{\pi}{\sqrt{\La}}$. When $\Si_0$ is non-orientable, its double cover $\ti{\Si}_0$ is then a two-side minimal surface of Morse index 1 by Theorem \ref{main theorem3}, so we finish the proof.


\section{Families of 1-cycles}\label{intro to 1-cycles}

In this section we define what we mean 
by a family of 1-cycles and 
a sweepout of a manifold by 1-cycles.

Following \cite{CC92}, \cite{NR04} for $k \in \mathbb{N}$ let $\Gamma_k(M)$
denote the space of all $k$-tuples $(\gamma^1,..., \gamma^k)$
of Lipschitz maps of $[0,1]$ to $M$ such that $\sum _i ^k \gamma^i(0) = \sum _i ^k \gamma^i(1)$
with the following topology. Using Nash embedding theorem we embed $M$ isometrically
into a Eucledian space and define the distance by the formula
$d_{\Gamma}((\gamma^1,...,\gamma^k),(\overline{\gamma^1},...,\overline{\gamma^k}))
= \max _{i,t} d_M \big(\gamma^i(t),\overline{\gamma^i}(t)\big)+ \sum \sqrt{\int_0 ^1 |\gamma^{i} \prime (t)-
\overline{\gamma^i} \prime (t)|^2dt}$.
We let $\Gamma = \bigcup \Gamma_k$.
Observe that the induced topology on $\Gamma_k$ is finer than the flat topology on the space
of integer 1-cycles \cite[\S 31]{Si83} and that the length functional is continuous on $\Gamma_k$. 

Let $\Gamma^0(M) \subset \Gamma(M)$ denote the space
of all constant curves (points).
Let $K$ be an $(n-1)-$polyhedral complex and $K_0$ be a 
subcomplex of $K$.
We say that a family of 1-cycles 
$\{z_t \}_{t \in K} \subset \Gamma_k(M)$ is a sweepout of $M$ if 

\begin{itemize}
\item For each $t \in K_0$ the cycle $z_t$ has zero length
\item $(\{z_t \}_{t \in K},\{z_t \}_{t \in K_0})$ 
is not contractible in
$(\Gamma, \Gamma^0)$
\end{itemize}

As noted above, constructing a family of cycles 
that is continuous in $\Gamma_k$ is a stronger result than 
constructing a continuous family of flat cycles.
Alternatively, we could start with a family
that is only continuous in the flat norm and then
use Almgren-Pitts theory of 
almost minimizing 1-dimensional varifolds as it was done
in \cite{P74} (see \cite{P81} for a much more general and
detailed exposition). 
The first step in this construction 
is to modify this family so that there is no cancellation 
of mass, that is, so that the family is
continuous both in the mass and the flat topology,
while controlling the mass of the maximal slice.
This is analogous to the procedure that we perform
in the proof of Lemma \ref{main_cycles1}.

Both approaches lead to a min-max sequence that
converges in the appropriate sense (in $\Gamma_k(M)$ topology in 
the first case or as varifolds for Almgren-Pitts theory)
to a stationary geodesic net. 

The min-max argument for families in $\Gamma_k(M)$
has the advantage of being simpler (see \cite{NR04} and Appendix of \cite{CC92}) and this is why we follow this approach.

It is often of interest to consider
families of cycles that arise as fibers
of a certain well-behaved mapping from
$M$ to a space of lower dimension.
For example, if $\Sigma$ is a 2-dimensional closed surface and $f: M \rightarrow [0,1]$
is an onto Morse function then we
consider the family $\{f^{-1}(t) \}_{t \in [0,1]}$.
For this family we have the following result.

\begin{lemma} \label{Morse to sweepout}
There exists a sweepout of $\Sigma$
by 1-cycles $\{z_t \}_{t \in [0,1]} \subset
\Gamma$, such that for each $t$ the image
of $z_t$ coincides with $f^{-1}(t)$
except possibly for a finite collection of 
points. 
\end{lemma}

\begin{proof}
Let $k$ be the maximum number of connected components
of $f^{-1}(t)$, $t \in [0,1]$. 
We define $\{z_t \} \subset \Gamma_k$ by induction 
on the number of singular points of $f$.

Since $f$ is an onto Morse function we have that
$f^{-1}(0)$ consists of $k_0 \leq k$ points $p_1,...,p_{k_0}$. 
Define the first $k_0$ components of $z_0 = (\gamma^1_0,..., \gamma^k_0)$ 
to be $\gamma^i([0,1])= p_i$ and for $i>k_0$ set $\gamma^i_0([0,1])= p_{k_0}$.
Let $t_1>0$  be the smallest critical value of $f$.
For each $i \leq k_0$ and $t<t_1$ we can define a homotopy 
$\{\gamma^i_t\}_{0 \leq t \leq 1}$,
so that $\gamma^i_t([0,1])$ is a connected component of $f^{-1}$
(we have $\gamma^i_t (0)= \gamma^i_t(1)$).

Let $t'$ be a critical point of $f$ and assume that $z_t$ is defined
for all $t<t'$. The singularity that occurs at $t'$ may be a
destruction/creation of a connected component of $f^{-1}(t)$
or a splitting/merging of two connected components. 
In the first case we proceed in the obvious way.
Consider the case of a splitting. Choose a small $\epsilon >0$
so that $f$ has no critical values in $[t'- \epsilon,t')$.
Since the number of connected components of $f^{-1}(t)$
is less than $k$ for $t \in [t'- \epsilon,t')$ there exists 
a constant component $\gamma^{k'}_t = p$ of $z_t$. 
Let $\gamma^{m}$ denote the component that splits into two
at time $t'$. For $t\in [t'- \epsilon,t'- \epsilon/2)$
we deform homotopically
$\gamma^{k'}$ to the point $\gamma^{m}_{t}(0)= \gamma^{m}_{t}(1)$.
For $t\in [t'- \epsilon/2,t)$ we homotop $\gamma^{m}$ and 
$\gamma^{k'}$ so that they form two arcs of the same connected component
of $f^{-1}(t)$ and their endpoints approach the singular point
of $f$ at $t'$. For $t \geq t'$ we can split the two arcs into 
two distinct connected components.
This ensure continuity of the family of cycles
in $\Gamma(M)$.
We deal with a merging of two components in a similar way.

This finishes the construction of a family of 
1-cycles $\{z_t \}_{t \in [0,1]} \subset
\Gamma(\Sigma)$ corresponding to $f$.
To see that this family is a sweepout 
recall a result of Almgren \cite{AF62} about homotopy groups of
the space of 1-cycles.

Let $Z_1(M, \mathbb{Z})$ denote the space of integer flat cycles in $M$.
Almgren constructed an isomorphism between homotopy groups
of the space of cycles
$\pi_k(Z_1(M, \mathbb{Z}), \{0\})$ 
and homology groups $H_{k+1}(M,\mathbb{Z})$ of the
space $M$.
Let $\Phi: \Gamma \rightarrow Z_1(M, \mathbb{Z})$ be the map that
sends each cycle in $\Gamma$ to the corresponding cycle in $Z_1(M, \mathbb{Z})$.
We will show that $\Phi(\{z_t \}, \{z_0,z_1\})$ is not contractible in the space of flat cycles 
and hence, it is not be contractible in $\Gamma(M)$.

Recall the definition of Almgren's map.
We pick a fine subdivision $t_1,...,t_n$ of $[0,1]$ and
for each $i$ consider a Lipschitz chain $c_i$ filling
$z_{t_i} - z_{t_{i-1}}$. The Almgren's map then
sends the family $\{z_t \}$ to the homology class
of $\sum_i c_i$. 
As long as the area of each filling $c_i$ is sufficiently
close to the minimal area, the exact choice of $c_i$
does not matter. Hence, by our construction of
$\{z_t \}$ it is immediately clear that 
$\sum_i c_i$ represents the generator of $H_2(\Sigma, \mathbb{Z})$.
\end{proof}

In this paper we construct two families of cycles
in a Riemannian 3-manifold of positive Ricci curvature:
the family of fibers of a mapping $f:M \rightarrow \mathbb{R}^2$
and the corresponding family $\{z_t \}$ of cycles in $\Gamma(M)$,
where $z_t$ and $f^{-1}(t)$ coincide except possibly 
for a finite number of constant curves.
Because of this correspondence we will often talk about them
as if they are the same family.

\section{Parametric sweepouts of surfaces}\label{parametric sweepouts of surfaces}

In this section we will prove a parametric version of
the following theorem of Balacheff and Sabourau \cite{BS10}.

\begin{theorem} \label{surface sweepout}
Let $\Sigma$ be a Riemannian surface
of genus $\gamma$, area $A$ and with a (possibly empty)
piecewise smooth boundary of length $L$. There 
exists a Morse function $f: \Sigma \rightarrow [0,1]$,
such that $f^{-1}(0) = \partial \Sigma$, and
the length of $f^{-1}(x)$ is at most
$C \sqrt{\gamma+1} \sqrt{A}+L$ for all $x \in [0,1]$ and
some universal $C<1000$.
\end{theorem}

This version of the theorem is slightly more general
and the upper bound for constant $C$ is better than in \cite{BS10}.
These improvements follow from the methods of \cite{L13},
\cite{GL} and \cite{L14}.

\begin{proof}
To each boundary component of $\Sigma$
we glue a very small disc to obtain a closed 
surface $\Sigma'$ of area $A+ \epsilon'$.

Let $\Sigma_0$ denote the unique surface 
of constant Gaussian curvature $-1$, $0$ or $1$,
which lies in the conformal class of $\Sigma$ 
and let $\phi: \Sigma_0 \rightarrow \Sigma'$ be a
conformal diffeomorphism.

For each $U \subset \Sigma'$ we will construct
a Morse function $f_U:U \rightarrow \mathbb{R}$,
such that $f^{-1}(0) = \partial U$,
and the length of $f^{-1}(x)$ is at most
$616 \sqrt{\gamma+1} \sqrt{Area(U)}+length(\partial U)$ 
for all $x \in [0,1]$.

\textbf{Step 1.} Choose $\epsilon>0$ much smaller than the injectivity 
radius of $\Sigma'$ and suppose $Area(U)< \epsilon^2$.
It follows from the systolic inequality on surfaces that 
the genus of $U$ is $0$. Let $\delta>0$ be a small number. 
By Lemma 19 from \cite{L13} if $\epsilon$ is
sufficiently small 
then there exists a Morse function 
$f: U \rightarrow [0,1]$ with $f^{-1}(0)=\partial U$
and the length of fibers at most $length(\partial U)+\delta$.

\textbf{Step 2.} 
Now we prove that for every
open subset $U \subset \Sigma$ there exists
a relative cycle $c$ (with $\partial c \subset \partial U$),
subdividing $U$ into two subsets of area
at least $\frac{1}{24}Area(U)$ and such that
$length(c) < 6.48 \max\{1, \sqrt{\gamma}\} \sqrt{Area(U)}$.
Let $r$ be the smallest radius such that
there exists a ball $B_r(x) \subset \Sigma_0$ (on the constant curvature conformal representative),
such that $Area_{\Sigma'}(\phi (B_r(x)) \cap U) = \frac{Area_{\Sigma'}(U)}{12}$.

We consider two cases. Suppose first that $r\leq 1$
then it it follows by comparison with a constant curvature space
that the annulus $B_{3r/2}(x) \setminus B_r(x)$ can be covered
by $10$ discs of radius $r$ in $\Sigma_0$.
Let $x$ be such that $Vol_{\Sigma'}(\phi( B_r(x)) \cap U)$
is maximazed.
It follows from the choice of $x$ and $r$ that
$Vol_{\Sigma'}(\phi(B_{3r/2}(x) \setminus B_r(x)) \cap U)
\leq \frac{10}{12} Area_{\Sigma'}(U)$.
Using the length-area method (cf. \cite{L14})
we find that there exists a cycle in the image of the annulus
of length 
$$ \leq \frac{1}{0.5} \sqrt{Area_{\Sigma_0}(B_{3r/2}(x)\setminus B_r(x))}
\sqrt{\frac{10}{12}Area_{\Sigma'}(U)}
\leq 4.12 \sqrt{Area_{\Sigma'}(U)}.$$

Now suppose $r\geq 1$. 
In this case we use an idea of \cite{CM08} (cf. proof of Lemma 3.3
in \cite{GL}) of considering systems of balls of radius $1$.
Let $k$ be the smallest integer, such that there exist
$x_1,...,x_k \in \Sigma_0$ with $Area_{\Sigma'}(\phi (\bigcup B_1(x_i))\cap U) \geq
\frac{Area_{\Sigma'}(U)}{24}$. Assume that $x_1,...,x_k$ are chosen in such a way that
this quantity is maximized. Observe that by the choice of $k$ we have
$Area_{\Sigma'}(\phi (\bigcup_{i=1}^k B_1(x_i))\cap U)< \frac{Area_{\Sigma'}(U)}{12}$.
For each $i=1,...,k$, let $B_1(y_i^1),...,B_1(y_i^{10})$
be a collection of $10$ balls of radius $1$ covering the 
annulus $B_{3/2}(x)\setminus B_1(x) \subset \Sigma_0$. 
By our choice of $x_i$ for each $j=1,...,10$ we have 
$Area_{\Sigma'}(\phi (\bigcup_{i=1}^k B_1(y_i^j))\cap U) \leq
Area_{\Sigma'}(\phi (\bigcup_{i=1}^k B_1(x_i))\cap U)< \frac{Area_{\Sigma'}(U)}{12}$.
As in the case $r \leq 1$ we use coarea formula and the length-area 
method to find a relative cycle $c$ in the image of $1/2-$neighbourhood
$\phi(\{x: 0<dist_{\Sigma_0}(x,\bigcup_{i=1}^k B_1(x_i))<1/2 \})\cap U$ 
of length at most 
$ \leq \frac{1}{0.5} \sqrt{Area_{\Sigma_0}(\Sigma_0)}
\sqrt{\frac{10}{12}Area_{\Sigma'}(U)}$.
Cycle $c$ subdivides $U$ into two parts each of area at least
$\frac{1}{24} Area_{\Sigma'}(U)$.
For a surface of genus $\gamma$ and constant curvature $1$, $0$ or $-1$
we have $Area_{\Sigma_0}(\Sigma_0)\leq 4 \pi \max\{1, \gamma -1\}$.
We conclude that the length of $c$ is bounded above by
$6.48 \max\{1, \sqrt{\gamma}\} \sqrt{Area_{\Sigma'}(U)}$.

\textbf{Step 3.} Let $U_1$ and $U_2$ be two open subsets
of $\Sigma$ with disjoint interiors and let $f_i: U_i \rightarrow 
[0,1]$, $i=1,2$, be Morse functions,
such that $f_i^{-1}(0) = \partial U_i$ and 
$length(f_i^{-1}(t)) \leq length(\partial U_i) + C$.
By Lemma 18 in \cite{L13}
there exists a Morse function $f: U_1 \cup U_2 \rightarrow [0,1]$
with $length(f_i^{-1}(t)) \leq length(\partial (U_1 \cup U_2))
+ 2length(\partial U_1 \cap \partial U_2)+ C$ and
$f^{-1}(0) = \partial (U_1 \cup U_2)$.

\textbf{Step 4.} To finish the proof we combine three steps above
in the following inductive argument. 
We claim that for each integer $0 \leq n \leq \log_{\frac{24}{23}} \frac{Area(\Sigma')}{\epsilon}$ and for every $U$ with 
$(\frac{24}{23})^{n-1} \epsilon<
Area(U) \leq (\frac{24}{23})^{n} \epsilon$
there exists a Morse function $f_U:U \rightarrow \mathbb{R}$,
such that $f^{-1}(0) = \partial U$,
and the length of $f^{-1}(x)$ is at most
$616 \sqrt{\gamma+1} \sqrt{Area(U)}+length(\partial U) + \delta$ 
for all $x \in [0,1]$.

By Step 1 the claim is true for $n=0$.
Suppose that every subset of $\Sigma'$ of area
at most $(\frac{24}{23})^{n-1} \epsilon$ satisfies
the inductive hypothesis. 
By Step 2 we can subdivide $U$ into two subsets
of area $\leq (\frac{24}{23})^{n-1} \epsilon$ by 
a cycle of length $6.48 \max\{1, \sqrt{\gamma}\} \sqrt{Area(U)}$.
By Step 3 there exists a desired Morse function with length of fibers
$\leq 616 \sqrt{\gamma+1} \sqrt{\frac{23}{24} Area(U)}+
2 * 6.48\sqrt{\gamma+1} \sqrt{Area(U)} +length(\partial U)+ \delta$.
$\delta$ can be chosen much smaller than $Area(U)$. This finishes the 
inductive argument.
\end{proof}


We state a parametric version of this result.

\begin{theorem} \label{thm:parametric}
Let $\Sigma$ be a surface of genus $\gamma$ and
let $\{g_t\}_{t \in [0,1]}$ be a smooth family of Riemannian metrics
on $\Sigma$, such that the area $A(\Sigma,g_t) \leq A$ for 
some constant $A$. There exists a continuous family of functions
$f_t: (\Sigma, g_t)  \rightarrow \mathbb{R}$, $t \in [0,1]$, such that
for each $x \in \mathbb{R}$ we have that
$f_t^{-1}(x)$ is a $1$-cycle in $(\Sigma,g_t)$
of length at most $2000 \sqrt{(\gamma +1) A}$.
All functions $f_t$ are Morse except possibly
for finitely many with a finite number of ordinary cusp singularities.
There exists a corresponding continuous family of sweepouts
in $\Gamma(\Sigma)$ (see Section \ref{intro to 1-cycles}).
\end{theorem}

Theorem \ref{thm:parametric} easily follows from the
following proposition conjectured by A. Nabutovsky
in a conversation with one of the authors.

\begin{proposition} \label{prop:homotopy}
Let $\Sigma$ be a closed Riemannian surface
and let $f_i:\Sigma \rightarrow \mathbb{R}$,
$i=0,1$, be two Morse functions, such that the length of $f_i^{-1}(x)$
is bounded above by $L$ for all $x$. Then $f_0$ and $f_1$
are homotopic through functions $f_t$, $0\leq t \leq 1$, with
the length  of $f_t^{-1}(x)$ bounded above by $2L+\epsilon$ for all $x$ and $t$ and 
arbitrarily small $\epsilon>0$.
Functions $f_t$ are all Morse, except for finitely many,
which may have a finite number of ordinary cusp singularities.
There exists a corresponding continuous family of sweepouts
in $\Gamma(M)$.
\end{proposition}

Note that cusp singularities occur in a generic 1-parameter
family of functions, so we can not expect them all to 
be Morse.

\begin{proof}

Without any loss of generality
we may assume that $f_i(\Sigma)=[0,1]$ 
for $i=0,1$. Let $z^i_s=f^{-1}_i(s)$, $s \in [0,1]$, be the 1-parametric family of
1-cycles given by the level sets of $f_i$.
Since $f_i$ is a Morse function we have that the family $\{z^i_s\}$ is a foliation
with finitely many singular leaves. The singularities are either constant curves
or curves with transverse self-intersections.

The set of singular points of $f_1$ and $f_2$
is finite. 
We can make finitely many local perturbations
so that the lengths of preimages changes by less than 
$\epsilon$ and the singular points of $f_1$ and $f_2$
are disjoint.
Moreover, by Thom's Multijet Transversality
we can make a perturbation so that each pre-image
of $f_1$
has only finitely many tangency points with preimages of $f_2$. 

To summarize, without any loss of generality
we may assume that families $z^i_{s}$ satisfy the
following:

\begin{enumerate}
\item If $z^1_{s'}$ is a singular leaf and $x$ is a singular 
point of $z^1_{s'}$ then it is disjoint from singular points
of $z^0_s$ for all $s$.

\item For each $s'$ all but finitely many $z^0_s$ intersect
$z^1_{s'}$ transversely; 
$z^0_s$ and $z^1_{s'}$ have at most one non-transverse touching
away from the singular points of $z^0_s$ and $z^1_{s'}$.
\end{enumerate}

Let $t \in [0,1]$. We define a 1-parameter family $z^t_s$, $s \in [0,1]$,
as follows. 

For $s \leq t$ we set $z^t_s=z^1_s$.
For $s>t$ we set $z^t_s = \partial (f_1^{-1}((-\infty,t]) \cup f_0^{-1}\big((-\infty, \frac{s-t}{1-t}])\big)$.
The family of cycles $\{ z^t_s \}$ is illustrated on Figure \ref{fig:parametric}.
Note that for each $s$ and $t$ the cycle $z^t_s$
consists of arcs of cycles $z^1_s$ and $z^0_{\frac{s-t}{1-t}}$.
Moreover, by properties (1) and (2) each $z^t_s$ is a collection of finitely many piecewise smooth curves with 
a finite number of corners. We can smooth out the corners by a small perturbation.
Using properties (1) and (2) we can perturb family $z^t_s$
so that it contains only finitely
many singular leaves.
Hence, we have constructed a continuous family of sweepouts by 
1-cycles in $\Gamma(\Sigma)$.

\begin{figure}
   \centering  
    \includegraphics[scale=0.4]{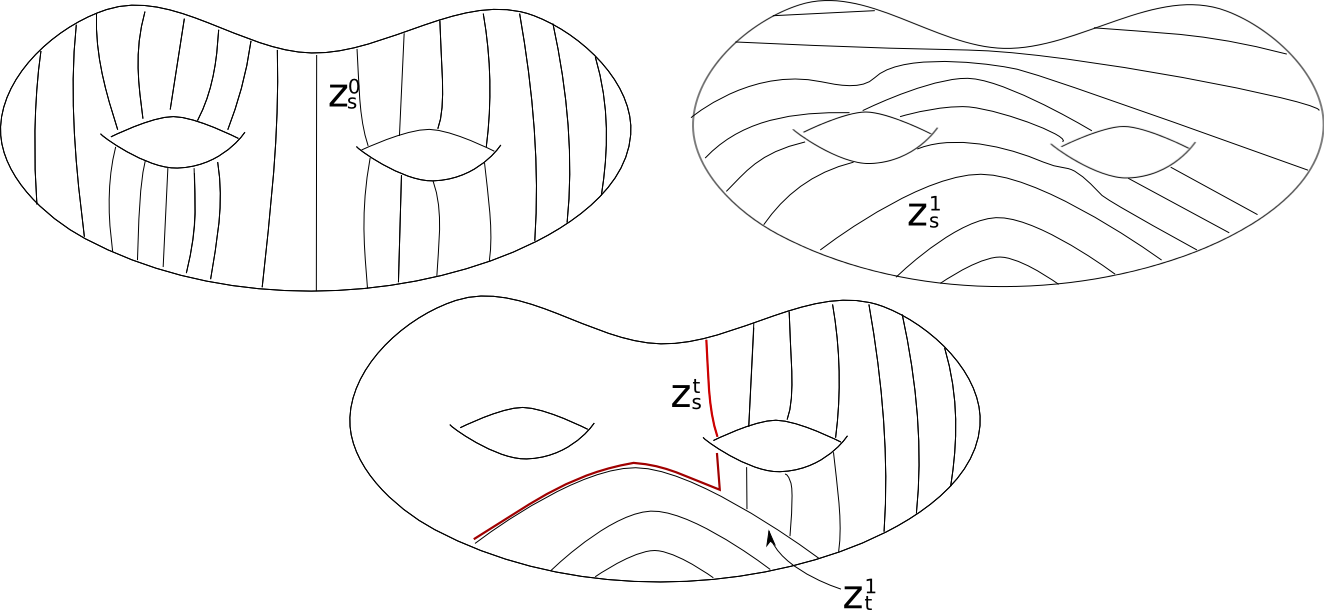}
    \caption{Homotopy through short sweepouts} \label{fig:parametric}
\end{figure}
 
Now we would like to construct a continuous family
of Morse functions with short fibers.
Note that the above family $\{z^t_s\}$ (for $t$ fixed)
is not a foliation because cycles 
$z^t_{s_1}$ and $z^t_{s_2}$ may overlap.
Also, we can not just glue together different fibers
of $f_0$ and $f_1$ because they may correspond to 
different values. 

We can fix this as follows.
Choose  $\epsilon_0>0$ sufficiently small,
so that every ball of radius $\epsilon_0$ 
is $(1+\epsilon)-$bilipschitz diffeomorphic
to a disc on the plane of the same radius. 
Let $\delta>0$ be such that cycles 
$f_0^{-1}(t)$ and $f_1^{-1}(t)$ have length 
less than $\epsilon_0$ for all $t \in [0,2\delta] \cup [1-\delta,1]$.

Let $t \in [\delta,1-\delta]$.
For $x \in f_1^{-1}([0,t])$ we set
$f_t(x)=f_1(x)$. 
Consider a foliation of $f_1^{-1}([t,1])$ by cycles
$z^{t+s} _{t + \frac{s}{\delta}(1-t)}$
for $s \in [0,\delta]$. This foliation will be
the level sets of $f_t$ on $f_1^{-1}([t,1])$.
We set $f_t(x)= t + s$ for
$x \in z^{t+s} _{t + \frac{s}{\delta}(1-t)}$. 
In particular, for $f_1^{-1}([t+\delta,1])$ we have 
$f_t(x)=t + \delta f_0(x)$.

For $t = \delta$
 we have that $f_{t}(x)= \delta + \delta f_0(x)$
everywhere except for a small set $U = f_1^{-1}([0,2\delta])$.
We linearly homotop $ f_{\delta}(x)$
to $\frac{f_{\delta}(x)-\delta}{\delta}$.
On $\Sigma \setminus U$ the function now coincides
with $f_0$. 
Since $U$ is contained in a collection of small
discs $(1+\epsilon)-$bilipschitz to discs in the plane,
we can homotop the function to $f_0$ in $U$ while controlling
the lengths of the fibers.

When $t=1-\delta$ function $f_t$ coincides with 
$f_1$ everywhere except for a small set $f_1^{-1}([1-\delta,1])$.
Similarly, we homotop the function to $f_1$ while controlling
the lengths of the fibers.

\end{proof}

We now prove Theorem \ref{thm:parametric}.
Let $C$ be the constant from Theorem \ref{surface sweepout}.
Fix a small $\epsilon >0$.
Subdivide $[0,1]$ into $n$ sufficiently small intervals $[t_i,t_{i+1}]$, 
so that for any $t_a, t_b \in [t_i,t_{i+1}]$ we have
$(1-\epsilon)^2 g_{t_a} \leq g_{t_b} \leq (1+\epsilon)^2 g_{t_a}$.

For each $i$ let 
$f_{t_{i}}: (\Sigma, g_{t_i}) \rightarrow \mathbb{R}$ be a Morse function
from Theorem \ref{surface sweepout} with fibers of length at most $C\sqrt{(\gamma+1)A}$. 
If $t \in [t_i,t_{i+1}]$ then we have that the functions
$f_{t_{i}}: (\Sigma, g_t) \rightarrow \mathbb{R}$
and $f_{t_{i-1}}: (\Sigma, g_t) \rightarrow \mathbb{R}$ have fibers
of length at most $(1+\epsilon)C \sqrt{(\gamma+1)A}$. 
By Proposition \ref{prop:homotopy} there exists a family of Morse functions
$\{h_r:\Sigma \rightarrow \mathbb{R}: r\in[0, 1]\}$, such that $h_0 = f_{t_{i-1}}$ and
$h_1= f_{t_{i}}$. Moreover, for any $r \in [0,1]$
the fibers of $h_r$ have length at most $2(1+\epsilon) C \sqrt{(\gamma+1)A}+ \epsilon
< 2000 \sqrt{(\gamma+1)A}$,
when measured with respect to $g_t$. 
We then set $f_t = h_{\frac{t-t_{i-1}}{t_{i}-t_{i-1}}}$ for $t \in [t_i,t_{i+1}]$.
This finishes the proof of Theorem \ref{thm:parametric}.


\section{Proof of Theorems \ref{MAIN_CYCLE} 
and Corollary \ref{SYSTOLIC}}\label{proof of main_cycle}

In this section we give a proof of Theorem \ref{MAIN_CYCLE}.

\begin{theorem} \label{main_cycles1}
Given a $3$-dimensional Riemannian manifold $M$ with an arbitrary metric of positive Ricci curvature, 
there exists a smooth map $f: M \rightarrow \mathbb{R}^2$
with fibers of length at most $C Vol(M)^{\frac{1}{3}}$.
The family of pre-images $\{f^{-1}(x) \}_{x \in \mathbb{R}^2}$ corresponds to
a continuous sweepout in $\Gamma(M)$.
\end{theorem}

\begin{proof}

Let $M$ be a 3-manifold of positive Ricci curvature.
By Theorem \ref{main theorem3} 
we have the following two possibilities:

Case 1. There exist a smooth function $f_0: M \rightarrow [-1,1]$, 
such that the fibers of $f$
for $t\in (-1,1)$ form a family of smooth diffeomorphic 
surfaces of genus $\gamma \leq 3$ and area
$\leq C vol(M)^{2/3}$ and $f_0^{-1}(-1)$ and $f_0^{-1}(1)$ are graphs.
Decomposition into cycles of controlled length then immediately follows by Theorem \ref{thm:parametric}.
The fact that these 1-cycle are continuous in $\Gamma(M)$
follows from the fact that $\{ \Sigma_t\}$ are continuous
in the smooth topology.

Now we consider Case 2 of Theorem \ref{main theorem3}.
Let $\Si_0 \subset M$ be a non-orientable min-max minimal surface
as in the theorem. Let $\gamma\leq 3$ be the genus of the double cover $\ti{\Sigma}_0$ of $\Sigma_0$. 
Let $ S_t = \{x \in M: dist(x,\Si_0) = t\}$ 
be the set of all points at a distance $t$ from $\Si_0$.
We have that for a sufficiently small $\delta >0$
and all $0<t\leq\delta$, the surface $ S_t$
is bi-Lipschitz diffeomorphic to the double cover $\ti{\Si}_0$
of $\Si_0$
with bi-Lipschitz constant $1+\epsilon$.
 Let $U = \Sigma_0 \cup \{S_t \}_{t \in (0, \delta]}$
denote the tubular neighborhood of $\Si_0$.

\begin{remark}
Suppose we are interested in constructing a function
$f: M \rightarrow \mathbb{R}^2$ with fibers forming
a family of 1-cycles of controlled length and continuous in
the \emph{flat norm}, but not necessarily  
continuous in $\Gamma(M)$.
Then we can argue as follows.

Let $f_0: \Sigma_0 \rightarrow [0,1] \times \{0\} \subset \mathbb{R}^2$
be the Morse function from Theorem \ref{surface sweepout}.
Composing with the covering map we obtain a map from the double cover 
$\tilde{f}_0: \tilde{\Sigma}_0 \rightarrow [0, 1] \times \{0\}$. 
Let $\ti{M}$
denote the manifold with boundary from Theorem \ref{main theorem3},
such that interior of $\ti{M}$ is isometric
to $M \setminus \Sigma_0$ and $\partial \ti{M} = \tilde{\Sigma}_0$.
By Theorem \ref{thm:parametric} we construct a function
$\ti{h}: \ti{M} \rightarrow [0,1]^2$, such that 
the restriction of $\ti{h}$ to $\partial \ti{M}$ is $\ti{f}_0$.
We then define $h(x) = \ti{h}(x)$ for all $ x\in M \setminus \Sigma_0$
and $h(x) = f_0(x)$ for $x \in \Sigma_0$.

In the remainder of the proof we will modify this construction
in order to produce a family, which is continuous in $\Gamma(M)$.
Our argument is similar to constructions of families 
of cycles with no cancellation of mass in the works
of Almgren and Pitts (see \cite{P81}).
\end{remark}

By Theorem \ref{main theorem3} there exists a sweepout of $M \setminus U$
by surfaces of controlled area. As in the first case
we construct a map $f:\overline{M \setminus U} \rightarrow [0,1] \times [0,1]$ 
with preimages of controlled length, and such that the preimages form a continuous family of $1$-cycles. Moreover,
we can do it in such a way so that $f$ restricted to $\partial U$ is a Morse function and
$\{f^{-1}(t,0)\}_{t \in [0,1]}$ is a family of $1$-cycles
sweeping out $\partial U$.

Next we construct an extension of this map to $U$.

\begin{lemma}
There exists a map $h: U \rightarrow [0,1] \times [-1,0]$, such that
the length of $h^{-1}(t, s)$, $(t, s)\in[0, 1]\times[-1, 0]$ is at most $10^4 \sqrt{Area(\Si_0)}$; $h(\cdot, 0): \partial U\rightarrow [0, 1] \times 0$ is a Morse function; and
the family of cycles $\{h^{-1}(t, s): (t, s)\in [0, 1]\times [-1, 0]\}$ is continuous in $\Gamma$.
\end{lemma}

\begin{proof}
Let $p: U \rightarrow \Si_0$
be the projection map, i.e. $p$ is the identity map on $\Si_0$ and
it sends $x \in S_t \subset U$ to the unique point $y \in \Si_0$ with $dist(x,y)=t$ for $t \in (0,\delta]$. 
Let $p_t$ denote the restriction of $p$ to $S_t$. Observe that $p_t$
is locally $(1+\epsilon)$-bi-Lipschitz.

By Theorem \ref{surface sweepout} there exists a Morse function
$g: \Sigma_0 \rightarrow \mathbb{R}$,
such that all preimages of $g$ are bounded in length by $1600 \sqrt{\gamma+1} \sqrt{A}$. 
We may assume that $g(\Sigma_0)= [0,1]$.
Let $c_s \subset \Sigma_0$ denote the $1$-cycle $g^{-1}(s)$.
For each $s \in [0,1]$ we will define a function 
$h_s: p^{-1}(c_s) \rightarrow [0,1]$ 
with fibers of controlled length.

Suppose first that $c_s$ is the pre-image of a regular value of $g$.
We can write $c_s$ as a union of finitely many
disjoint embedded circles $c_s = \bigsqcup c_s^i$ in $\Sigma_0$.

The set $p^{-1}(c_s^i)$ is diffeomorphic to a cylinder
or a Mobius band. We will construct a foliation of $p^{-1}(c_s^i)$
by 1-cycles and use it to define function $h$.
Choose two distinct points $x$ and $y$ on $c_s^i$
and consider a small tubular neighborhood $V_x$ (resp. $V_y$)
of $p^{-1}(x)$ (resp. $p^{-1}(y)$) in $p^{-1}(c_s^i)$.
We foliate $V_x$ by 1-cycles as depicted on Figure \ref{fig:foliation}(a).
Clearly we can define a smooth function from $V_x$ to $[0,1]$, whose fibers
are 1-cycles in the foliation and which is non-degenerate everywhere
except for a saddle point at $x$. Call this type of 
foliation of $V_x$ a {\em saddle foliation}. Similarly,  
Figure \ref{fig:foliation}(b) depicts a foliation of
$V_y$ and we define the corresponding function from $V_y$ to $[0,1]$
with a singularity of index $2$ at $y$ (a maximum point for $h_s$). Call this type of foliation
of $V_y$ as {\em node foliation}.
We can extend $V_x$ and $V_y$ so that they cover all of $p^{-1}(c_s^i)$
and extend the corresponding foliations and functions in the obvious way.
Observe that the lengths of the preimages are bounded above
by $2 length(c^i_s)+O(\epsilon)$.
For $u \in p^{-1}(c_s)$ we define the function $h(u) = (s,-h_s(u)) \subset [0,1] \times [-1,0]$.
Hence, we described a construction of $h$ on each connected component of
$p^{-1}(c_s)$ in the case when $c_s$ is non-singular.
As $s$ varies we can vary $x$ and $y$ and the corresponding
foliations continuously and extend the map $h$
to $p^{-1}(g^{-1}([s,s+a)))$, where $s+a$ is the first singular value of $g$ after $s$.

\begin{figure}
   \centering  
    \includegraphics[scale=0.4]{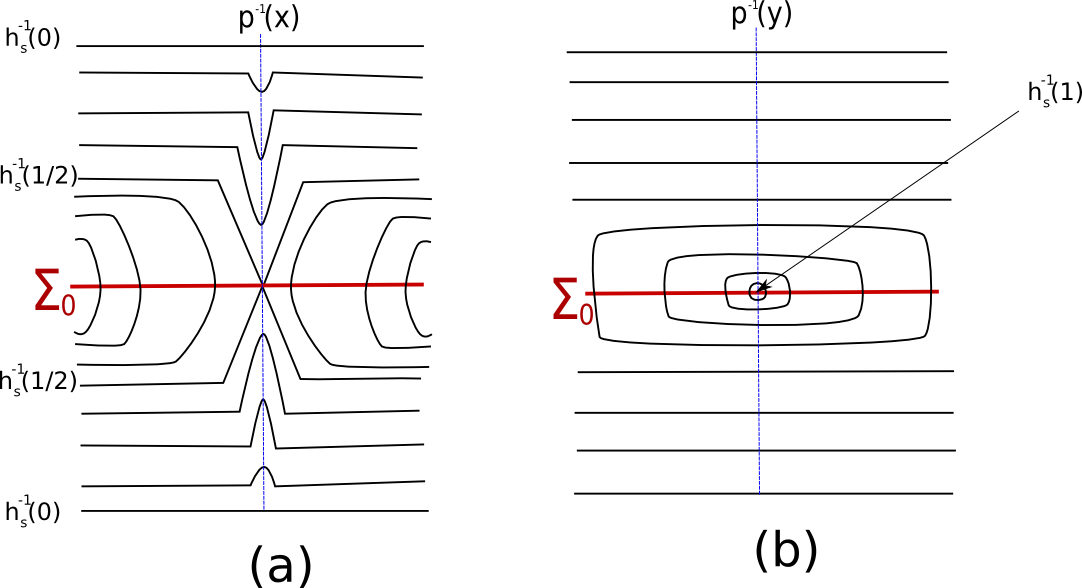}
    \caption{Foliation of $V_x$ and $V_y$ by fibers of function $h_s$.} \label{fig:foliation}
\end{figure}

Let $c_{s}^i$ be a singular connected component of $c_{s}$.
Consider the case when $c_{s}^i$ is a point.
This occurs when we have a creation or a destruction of a connected component.
Suppose first that $c_{s+\epsilon}$ has one less connected components than
$c_{s-\epsilon}$. The case of
creation of a connected component is 
treated similarly.
We modify the function $h$ on a neighborhood $V=p^{-1}(\{c^i_{s'}\}_{s' \in [s-\epsilon,s]}) \subset U$ as follows. 
Let $D_\epsilon \subset \mathbb{R}^2$ denote a disc of radius $\epsilon$. Observe that there exists a diffeomorphism $\phi: D_{\epsilon} \times [-\delta,\delta] \rightarrow V$,
such that for a concentric circle $S_r \subset  D_{\epsilon}$ of radius $r$ and $ t \in (0,\delta]$
we have
$\phi (S_r \times \{t,-t\}) = \Sigma_t \cap p^{-1}(c_{s-r}^i)$.
Let us consider the set $\phi(S_r \times \{t\})$. 
For $|t|\geq \frac{\delta}{\epsilon} r$ we define $h$ on $\phi(S_r \times \{t\})$
to be equal to $(s-r,\frac{|t|}{\delta}-1)$. In particular, if $x\in p^{-1}(c_{s}^i)$ and
$x \neq c_{s}^i$ then $h^{-1}(h(x))$ consists of two points.
For $|t|< \frac{\delta}{\epsilon} r$ we construct $h(x)$
exactly as how we constructed $h_s$ on the pre-image under $p$ of a non-singular
curve, but we scale the image so that $h$ is continuous along the diagonal
$|t|= \frac{\delta}{\epsilon} r$.
We do it as follows. Fix $r\in [0, \epsilon)$ and let $W = \phi(S_r \times (-\frac{r}{\epsilon}\delta,\frac{r}{\epsilon}\delta))$.
Construct a foliation of $W$ by 1-cycles with two singularities as on Figure \ref{fig:foliation}.
We then define $h$ on $W$ so that $h(W)=\{s-r\} \times [-1,-1+\frac{r}{\epsilon})$ 
and its preimages are given by 1-cycles in the foliation.

Suppose now that $c_{s+a}^i$ is a figure-8 curve with a self-intersection at a point $z \in c_{s+a}^i$.
This means that we either have a splitting of one component into two or a merging of two components
into one
(for otherwise we could perturb this family of cycles so that no singularity would occur).
We consider the case of two components merging and the argument for the other case is analogous.
For $s'<s+a$ let $c_{s'}^i$ and $c_{s'}^{i+1}$ be two components
that merge into $c_{s+a}^i$ at time $s+a$. 
We arrange $h_{s'}$ on $p^{-1}(c_{s'}^j)$ ($j=i,i+1$) so that 
$h_{s'}$ takes on its maximum at a point $y_j(s')$ on $p^{-1}(c_{s'}^j)$ with $y_j(s')$ 
converging to the self-intersection point
of the figure-8: $y_j(s') \rightarrow z$ as $s' \rightarrow s+a$.
For $p^{-1}(c_{s+a}^i)$ we obtain two node foliations (Figure \ref{fig:foliation}(b))
glued along $p^{-1}(z)$.
Observe that we can arrange the foliations to match properly so that they correspond to preimages of a smooth function on
$p^{-1}(c_{s+a}^i)$ and, moreover, we can extend
it to a foliation of $p^{-1}(c_{s+a-\epsilon}^i)$ by cycles satisfying
$2 length(c_{s+a-\epsilon}^i)+O(\epsilon)$ upper bound on their lengths.
This foliation has two node foliations and two saddle foliations.
We can arrange for one node foliation and one saddle foliation to 
collide and annihilate in a saddle-node bifurcation.
This can be done without increasing the length bounds by more than $O(\epsilon)$.
(Thinking of $h_s$ as a height function one can picture saddle-node bifurcation
as smoothing out a hill). This completes the construction of $h$. 
It is clear from the construction that the corresponding family
of cycles is continuous in $\Gamma(M)$.
\end{proof}

From the construction in the proof it follows that $\{h^{-1}(s,0)\}$, $s \in [0,1]$, is a family of $1$-cycles
sweeping out $\partial U$. Moreover $h$
has only finitely many singularities on $\partial U$,
all of them non-degenerate.
By Proposition \ref{prop:homotopy}
$\{h^{-1}(s,0)\}$ and $\{f^{-1}(s,0)\}$ can be connected by a family of sweepouts
of controlled length.
After a small perturbation this produces the desired map from $M$ to $\mathbb{R}^2$
with fibers of controlled length.
\end{proof}

We now use Theorem \ref{MAIN_CYCLE} 
to give alternative proofs of results of Gromov and Nabutovsky-Rotman
stated in Corollary \ref{SYSTOLIC}.
First, let $M \cong \mathbb{R}P^3$. By Theorem \ref{MAIN_CYCLE} there exists a 
function $f: M\rightarrow [0,1]^2$ and a continuous
sweepout $\{ z_t= f^{-1}(t)\}_{t \in [0,1]^2} \subset \Gamma$ 
of $M$ by 1-cycles of length at most $C Vol(M)^{\frac{1}{3}}$.
Corollary \ref{SYSTOLIC} follows from the 
lemma below.

\begin{lemma}
A connected component of some cycle $z_t$ in the sweepout 
that we constructed 
is a non-contractible loop in $M$.
\end{lemma}

\begin{proof}
The idea of the proof of this lemma was suggested to us
by Nabutovsky and Rotman (see also \cite{GZ}, where a version of this
lemma is proved for sweepouts of 2-dimensional tori).
For contradiction we assume that every connected component of
$z_t$ is contractible for all $t$.

Let $f:M \rightarrow [0,1]^2$ be the map
constructed in Theorem \ref{MAIN_CYCLE}.
From \cite[p.128]{Gro83} recall the definition of a connected map
$ \overline{f}:M \rightarrow X$
associated to $f$. The set $X$ is defined as the quotient
of $M$ by the equivalence relation 
$x \sim y$ if $x$ and $y$ are in the same connected component 
of $f^{-1}(t)$ for some $t \in [0,1]^2$ and $ \overline{f}:M \rightarrow X$
is the quotient map. There is a unique map 
$\tilde{f}: X \rightarrow [0,1]^2$, such that
$\tilde{f} \circ \overline{f}=f$.
Observe that $X$ must be connected
since $M$ is. Also, by our construction of $f$
we can endow $X$ with the structure of a polyhedral complex,
such that for an interior point $x$ of every face we have 
that every pre-image
$\overline{f}^{-1}(x)$ is a simple closed curve in $M$ and
for each point $x$ contained in the 1-skeleton of $X$ 
the pre-image $\overline{f}^{-1}(x)$ is a point or a closed
curve with a finite number of self-intersections.

We claim that $X$ is simply-connected. Indeed,
since $\overline{f}$ is a connected map
every loop in $X$ lifts to a loop in $M$,
so the induced map on the fundamental groups
is surjective. On the other hand,
the cohomology ring structure of $\mathbb{R}P^3$
forces the induced map on homology to be trivial.

Let $\overline{M}$ denote the universal cover
of $M$ and $p: \overline{M} \rightarrow M$ be the covering map.
Consider the composition $F= \overline{f} \circ p$.
By our assumption for each $x$ we have that
$\overline{f}^{-1}(x)$ is a contractible
closed curve in $M$,
so $F^{-1}(x)$ consists of $k$ disjoint closed curves in $\overline{M}$.
Let $\overline{F}:\overline{M} \rightarrow \overline{X}$ be
the connected map associated to $F$.
We obtain that $\overline{X}$ is a $k-$fold covering space for $X$.
This contradicts the fact that $X$ is simply-connected.
\end{proof}

Now by applying Brikhoff curve-shortening process we obtain
a closed geodesic in $M$ of length at most $C Vol(M)^{\frac{1}{3}}$
proving Corollary \ref{SYSTOLIC} for the case when $M \cong mathbb{R}P^3$.

If $M$ is homeomorphic to a 3-sphere then by the min-max argument described  in 
\cite{NR04} existence of a sweepout, which is continuous in $\Gamma(M)$ 
by short 1-cycles implies existence of a geodesic net with the desired length bound.


\section{Further discussion}\label{open questions}

The first open question is the relation between the Almgren-Pitts min-max minimal surface and the Simon-Smith (also Colding-De Lellis) min-max minimal surface in $(S^3, g)$ with positive Ricci curvature (or even in any 3-manifold with positive Ricci curvature). The Almgren-Pitts minimal surface, which we use in this paper, has area bounded by the $\frac{2}{3}$-power of the volume up to a universal constant, and genus $\leq 3$; while the Simon-Smith min-max minimal surface has genus $0$ (Heegaard genus for $S^3$) but no a priori area bound in terms of the volume of the ambient manifold. It is then a natural question to compare them.

The second open question is whether we could have diameter bound for the whole min-max family constructed in Theorem \ref{main theorem3} when we assume the scalar curvature lower bound instead of just getting a diameter bound for the min-max surface as in Theorem \ref{main theorem3.1}. 

These two questions are related to the problem of finding an upper
bound for the length of the shortest non-trivial closed geodesic in manifold $(S^3, g)$.
The methods of this paper
produce a sweepout of $(S^3, g)$ by short 1-cycles which yields a stationary geodesic net in
$(S^3, g)$ of controlled length, but they do not give any bound for the length of the shortest
closed geodesic. For this purpose one would need to consider
sweepouts by loops instead of 1-cycles. If a manifold $(S^3, g)$ admits a sweepout by 2-spheres or 2-tori
of controlled area $A$ and diameter $d$ then it seems plausible that using methods of
\cite{LNR} one could bound the length of the shortest closed geodesic in $(S^3, g)$ in
terms of $A$ and $d$.


\begin{tabbing}
\hspace*{7.5cm}\=\kill
Yevgeny Liokumovich          \> Xin Zhou\\
Department of Mathematics    \> Department of Mathematics\\
Imperial College London      \> Massachusetts Institute of Technology\\
South Kensington Campus      \> 77 Massachusetts Avenue \\
London SW7 2AZ, UK           \> Cambridge, MA 02139-4307, USA\\
Email: y.liokumovich@imperial.ac.uk  \> Email: xinzhou@math.mit.edu 
\end{tabbing}

\end{document}